\documentclass[12pt,a4paper]{article}
\usepackage[english]{babel}
\usepackage{fancyhdr}
\usepackage[latin1]{inputenc}
\usepackage{amsmath}
\usepackage{amsfonts}
\usepackage{amssymb}
\usepackage{graphicx}
\usepackage{latexsym}
\usepackage{amsthm}
\usepackage{color}
\usepackage{midpage}
\usepackage[width=15.00cm]{geometry}
\usepackage{amscd}
\usepackage{tikz-cd}
\usepackage{hyperref}

\numberwithin{equation}{section}
\theoremstyle{plain}
\newtheorem{theorem}{Theorem}[section]
\theoremstyle{corollary}

\theoremstyle{lemma}
\newtheorem{lemma}[theorem]{Lemma}
\theoremstyle{proprosition}
\newtheorem{proposition}[theorem]{Proposition}
\theoremstyle{assumption}
\newtheorem{assumption}[theorem]{Assumption}
\theoremstyle{condition}

\theoremstyle{definition}

\theoremstyle{example}

\newtheorem{remark}[theorem]{Remark}

\input{tcilatex}
\begin{document}

\title{{\Huge Flexible-bandwidth Needlets}}
\author{Claudio Durastanti \\
\textit{Dipartimento SBAI, La Sapienza Università di Roma} \and Domenico
Marinucci \\
\textit{Dipartimento di Matematica, Università di Roma Tor Vergata} \and %
Anna Paola Todino \\
\textit{Dipartimento di Matematica, Politecnico di Torino} }

\date{}
\maketitle

\begin{abstract}
We investigate here a generalized construction of spherical
wavelets/needlets which admits extra-flexibility in the harmonic domain,
i.e., it allows the corresponding support in multipole (frequency) space to
vary in more general forms than in the standard constructions. We study the
analytic properties of this system and we investigate its behaviour when
applied to isotropic random fields: more precisely, we establish asymptotic
localization and uncorrelation properties (in the high-frequency sense)
under broader assumptions than typically considered in the literature.
\end{abstract}

\begin{itemize}
\item \textbf{Keywords and Phrases:} Spherical wavelets, needlets, spherical
random fields, high-frequency asymptotics.

\item \textbf{AMS Classification:} 60G60; 62M40, 42C40.
\end{itemize}

\section{Introduction}

\label{sec:introduction}The statistical analysis of spherical random fields
has become a rather important research topic in the last 15 years. In
particular, strong motivations have come from a variety of fields, most
notably Cosmology and Astrophysics, Geophysics, Climate Sciences: at the
same time, it has become clear that the analysis of spherical data can lead
to a number of deep mathematical issues, which have independent interest (see \cite{Gneiting,Lang,Moller,Ziegel} and the references therein).
Among these issues, a very important role has been played by the
investigation of spherical wavelet systems, and the analysis of their
properties when applied to spherical random fields.

Among spherical wavelets, one of the most successful proposals is certainly
the needlet system, which was introduced by \cite{npw1,npw2} and then
applied to random fields and cosmological data immediately after by \cite
{bkmpAoS, MNRAS2008, PBM2006}; extensions to more general
harmonic kernels were discussed by \cite{gm2}. Needlets on one hand
represent a tight-frame system and hence satisfy classical requirements of
approximation theory; on the other hand under some regularity conditions
needlet coefficients have been shown to enjoy asymptotic uncorrelation
properties (in the high-resolution sense) which makes their application to
random fields extremely powerful. Extensions of the needlet construction to
more general homogeneous spaces of compact groups were given for instance by
\cite{Coulhon,GellerPesenson,KerkyacharianPetrushev}; statistical
applications are currently too many to be recalled in any reasonable
completeness: we refer for instance to \cite{KerkyacharianPhamPicard,KerkyacharianNicklPicard} or more recently \cite
{BD17, Cheng, Claudio16,  dur,LeGia, ST,Wang, Fan, Olenko}. Applications in Cosmology and Astrophysics are
discussed for instance in \cite{carron,mcewen,Oppizzi, PlanckIS, WangSloan,
wiaux2}) and the references therein.

Our purpose in this paper is to generalize the needlet construction,
allowing for a much more flexible form of the kernel function in the
harmonic domain; we then proceed to investigate the properties of these
generalized needlet transforms when applied to isotropic spherical random
fields. In particular, we establish explicit bounds on the decay of the
correlation function for needlet coefficients under much broader conditions
than given in the existing literature; these results make possible
asymptotic statistical inference in the high-frequency sense for a much
greater family of random models. In order to make these statements more
precise, however, we need first to review some notation and background
results.

\subsection{Some Background Results and Notation}

Let us recall first some standard background material on harmonic analysis
on the sphere; we refer for instance to \cite{atki,MaPeCUP} for more
discussion and details. We write as usual $L^{2}\left( \mathbb{S}^{d}\right)
$ to denote the space of square-integrable functions on the sphere (with
respect to Lebesgue measure), and $\omega _{d}=\frac{2\pi ^{\frac{d+1}{2}}}{
\Gamma \left( \frac{d+1}{2}\right) }$ to denote the $d$-dimensional
spherical surface measure, with $\Gamma (\cdot )$ the usual Gamma function.
The following decomposition holds:
\begin{equation*}
L^{2}\left( \mathbb{S}^{d}\right) =\bigoplus_{\ell =0}^{\infty }\mathcal{H}
_{\ell ;d}\text{ },
\end{equation*}
where $\mathcal{H}_{\ell ;d}$ is the restriction to $\mathbb{S}^{d}$ of the
space of harmonic and homogeneous polynomials of degree $\ell $ on $\mathbb{R
}^{d+1}$. The spaces $\mathcal{H}_{\ell ;d}$ have dimension
\begin{equation*}
N_{\ell ;d}=\frac{\ell +\eta _{d}}{\eta _{d}}\binom{\ell +2\eta _{d}-1}{\ell
}=\frac{2\ell ^{d-1}}{\left( d-1\right) !}\left( 1+o_{\ell }\left( 1\right)
\right) ,\quad \eta _{d}=\frac{\left( d-1\right) }{2}\text{ ;}
\end{equation*}
the elements $\left\{ f_{\ell }\in \mathcal{H}_{\ell ;d}\right\} $ are the
eigenfunctions of the Laplace-Beltrami operator
\begin{equation*}
\Delta _{\mathbb{S}^{d}}f_{\ell }=-\ell (\ell +d-1)f_{\ell }\text{ , }\ell
=0,1,2,...
\end{equation*}
On $\mathcal{H}_{\ell ;d},$ we can choose a (real or complex-valued)
orthonormal basis, which we write as\\ $\left\{ Y_{\ell, m}:m=1,\ldots ,N_{\ell
;d}\right\} ,$ omitting the dependence on the dimension $d.$ More
explicitly, this entails that every function $f\in L^{2}\left( \mathbb{S}
^{d}\right) $ admits the expansion
\begin{equation}
f\left( x\right) =\sum_{\ell \geq 0}\sum_{m=1}^{N_{\ell ;d}}a_{\ell,
m}Y_{\ell ,m}\left( x\right) ,  \label{eq:exp}
\end{equation}
where, for $\ell \geq 0$ and $m=1,\ldots ,N_{\ell ;d}$,
\begin{equation*}
a_{\ell, m}=\int_{\mathbb{S}^{d}}\overline{Y}_{\ell, m}\left( x\right) f\left(
x\right) dx\in \mathbb{C}\text{ },
\end{equation*}
are the so-called spherical harmonic coefficients. For any choice of an
orthonormal basis, the following \emph{addition formula} holds
\begin{eqnarray*}
Z_{\ell ;d}\left( x_{1},x_{2}\right)  &=&\sum_{m=1}^{N_{\ell ,d}}\overline{Y}
_{\ell, m}\left( x_{1}\right) Y_{\ell, m}\left( x_{2}\right)  \\&=& \frac{\ell +\eta _{d}}{\eta _{d}\omega _{d}}G_{\ell }^{\left( \eta
_{d}\right) }\left( \left\langle x_{1},x_{2}\right\rangle \right) ,\quad
\text{for }x_{1},x_{2}\in \mathbb{S}^{d},  
\end{eqnarray*}
where $\left\langle \cdot ,\cdot \right\rangle $ denotes the standard scalar
product over $\mathbb{R}^{d+1}$, $G_{\ell }^{\left( \eta _{d}\right) }$ is
the Gegenbauer polynomial of degree $\ell $ and parameter $\eta _{d}$ (see
\cite{atki}, Chapter 2); with some abuse of notation, we shall write both $
Z_{\ell ;d}\left( x,y\right) $ or $Z_{\ell ;d}\left( \langle x,y\rangle
\right) $, depending on the context. For instance, for $d=2$ we have
\begin{equation*}
Z_{\ell ;2}\left( x_{1},x_{2}\right) =\frac{2\ell +1}{4\pi }P_{\ell }\left(
\left\langle x_{1},x_{2}\right\rangle \right) \text{ ,}
\end{equation*}
where
\begin{equation*}
P_{\ell }\left( \cdot \right) :[-1,1]\rightarrow \mathbb{R}\text{ , }P_{\ell
}\left( t\right) :=\frac{d^{\ell }}{dt^{\ell }}(t^{2}-1)^{\ell }\text{ , }
\ell =0,1,2,...
\end{equation*}
is the usual Legendre polynomial.

The following \emph{reproducing kernel property} holds
\begin{equation*}
\int_{\mathbb{S}^{d}}Z_{\ell ;d}\left( \langle x,y\rangle \right) P_{\ell
^{\prime };d}\left( \langle y,z\rangle \right) dy=Z_{\ell ;d}\left( \langle
x,z\rangle \right) \delta _{\ell ^{\prime }}^{\ell },\text{ for all }\ell
,\ell ^{\prime }\in \mathbb{N}\text{ ,}  
\end{equation*}
where $\delta _{\cdot }^{\cdot }$ is the Kronecker delta. Clearly for any $
f\in L^{2}\left( \mathbb{S}^{d}\right) $ its projection over the space $
\mathcal{H}_{\ell ;d}$ is given by
\begin{equation*}
f_{\ell }\left( x\right) =Z_{\ell ;d}\left[ f\right] \left( x\right) =\int_{
\mathbb{S}^{d}}Z_{\ell ;d}\left( \langle x,y\rangle \right) f\left( y\right)
dy=\sum_{m=1}^{N_{\ell ,d}}a_{\ell ,m}Y_{\ell ,m}\left( x\right) .
\end{equation*}%
The standard needlet kernel, as introduced by \cite{npw1}, can then be
defined as follows; for any $j=1,2,...$
\begin{equation*}
\Psi _{j}\left( x,y\right) =\sum_{\ell \geq 0}b\left( \frac{\ell }{B^{j}}
\right) Z_{\ell ,d}\left( \left\langle x,y\right\rangle \right) ,
\end{equation*}
where $B>1$ is a fixed (bandwidth) parameter, whereas $b(\cdot ):\mathbb{R}
\rightarrow \mathbb{R}$ is a weight function which satisfies three
properties: a) it is compactly supported in $[\frac{1}{B},B];$ b) it is $
C^{\infty };$ c) it satisfies the \emph{Partition of Unity} property,
namely, $\sum_{j\geq 1}b^{2}(\frac{\ell }{B^{j}})=1$, for all $\ell \in
\mathbb{N}.$

Under these conditions, in \cite{npw1} the following nearly-exponential
localization property is established; for all $x,y\in \mathbb{S}^{2}$ and
for all integers $M,$ there exists a constant $C_{M}$ (depending on $b(\cdot
),$ but not on $x,y$ or $j)$ such that
\begin{equation}
\left\vert \Psi _{j}\left( x,y\right) \right\vert \leq \frac{C_{M}B^{dj}}{
\left\vert 1+B^{j}d_{\mathbb{S}^{2}}(x,y)\right\vert ^{M}}\text{ ,}  \label{StanLocProp}
\end{equation}
where $d_{\mathbb{S}^{2}}(x,y):=\arccos (\left\langle x,y\right\rangle )$ is
the standard geodesic distance on the sphere.  This key localization
property can then be exploited to derive a number of extremely important
features of the needlet system; indeed the needlet projectors are simply
defined by
\begin{equation}
\psi _{j,k}\left( x\right) =\sqrt{\lambda _{j,k}}\sum_{\ell \geq 0}b\left(
\frac{\ell }{B^{j}}\right) Z_{\ell ,d}\left( \left\langle x,\xi
_{j,k}\right\rangle \right) ,  \label{needproj}
\end{equation}
where $\left\{ \xi _{j,k}:j\geq 0,k=1,\ldots ,K_{j}\right\} $ and $\left\{
\lambda _{j,k}:j\in \mathbb{N},k=1,\ldots ,K_{j}\right\} $ are cubature points and
weights respectively, see also \cite{npw1}. The corresponding needlet
coefficients are defined as
\begin{equation}
\beta _{j,k}=\langle f,\psi _{j,k}\rangle _{L^{2}\left( \mathbb{S}^{d}\right)
},\quad j\geq 0,k=1,\ldots ,K_{j},  \label{need_coef_disc}
\end{equation}
where $f(\cdot )$ denotes any (random or deterministic) function in $L^{2}(
\mathbb{S}^{d}).$

As mentioned above, a key ingredient for the interest that needlet
transforms have drawn when applied to the analysis of spherical random
fields are their asymptotic uncorrelation properties. We can recall them
briefly as follows. Assume we have a zero-mean, finite variance, isotropic
random field on $\mathbb{S}^{d};$ then the spectral representation (\ref
{eq:exp}) holds in the $L^{2}(\Omega \times \mathbb{S}^{d})$ sense, where the
family of zero-mean random coefficients $\left\{ a_{\ell, m}\right\}_{\ell\in\mathbb{N},m=1,\ldots,N_{\ell;d}}$  satisfies
\begin{equation*}
\mathbb{E}\left[ a_{\ell, m}\overline{a}_{\ell ^{\prime },m^{\prime }}\right]
=\delta _{\ell }^{\ell ^{\prime }}\delta _{m}^{m^{\prime }}C_{\ell }\text{ ,
}\ell,\ell^{\prime} \in \mathbb{N},m,m^{\prime}=1,\ldots,N_{\ell;d} .
\end{equation*}
The sequence $\left\{ C_{\ell }\right\}_{\ell \in \mathbb{N}} $ is labelled the angular power
spectrum of the random field. In \cite{bkmpAoS} and many subsequent papers
(starting from \cite{BKMPBer,mayeli}), it is assumed that the angular power spectrum obeys some
regularity condition such as
\begin{equation}
C_{\ell }=g(\ell )\ell ^{-\alpha },\text{ }\alpha >2\text{ , \ some positive
}g\in C^{\infty }\text{ such that }g^{(r)}(u)=O(u^{-r})\text{ , as }
u\rightarrow \infty \text{ .}  \label{power1}
\end{equation}
For instance, $g(\cdot)$ could be any slowly-varying function, in the sense of
\cite{BinghamGoldieTeugels}. Now write
\begin{equation}
\beta _{j}(x):=\int_{\mathbb{S}^{d}}f(y)\Psi _{j}\left( x,y\right) dy\text{ ;
}  \label{need_coef_cont}
\end{equation}
up to a normalization, (\ref{need_coef_cont}) can be simply interpreted as a
continuous version of (\ref{need_coef_disc}): note indeed that $\beta
_{j,k}=\beta _{j}(\xi _{j,k})\sqrt{\lambda _{j,k}}$. Assuming that $\left\{
f(\cdot )\right\} $ is an isotropic spherical random field whose angular
power spectra satisfies (\ref{power1}), it was shown in \cite{bkmpAoS} that
for all positive integers $N$ there exists $C_{N}>0$ such that
\begin{equation}
\mathrm{Corr}(\beta _{j}(x),\beta _{j}(y))\leq \frac{C_{N}}{(1+B^{j}d_{\mathbb{S}
^{2}}(x,y))^{N}}\text{ for all }j\in \mathbb{N}.  \label{corr}
\end{equation}
In words, (\ref{corr}) is
stating that for any two fixed points on the sphere, the correlation between
the standard needlet transforms of order $j$ at these two points is going to
zero nearly-exponentially (i.e., faster than any polynomials) as $j$
diverges. This uncorrelation property is equivalent to high-frequency
independence in the Gaussian case, and hence it makes possible the
implementation of a number of statistical procedures whose properties can be
rigorously established, in the high-frequency regime.

\subsection{Main Results}
As mentioned above, our plan in this paper is to introduce a further degree
of flexibility in the needlet construction, by allowing the scale width in
the multipole space to cover a much broader spectrum of possibilities than
in the existing literature. More precisely, as illustrated in the previous
Section in the standard needlet construction the $j-$order transform is
supported in the harmonic space over the interval $(B^{j-1},B^{j+1})$. There are several reasons, we believe, why it is of interest to consider
needlet-like transforms with more general support in the harmonic domain.
For instance, practitioners may be interested in multipoles ranging over
more general domains than $\ell \in (B^{j-1},B^{j+1})$ for physical reasons
related to their model of interests; otherwise, experimental settings may
put specific constraints on the multipoles on which needlet transforms can
be computed. Also, the range of values $(B^{j-1},B^{j+1})$ may simply be
considered to grow too rapidly for large values of $j$, and data
analysts/applied scientists may prefer to reduce it to achieve better
frequency-domain resolution in their analysis. These situations have
actually taken place, for instance, in the analysis of Cosmological data,
and it has been common to implement needlets on various multipole windows,
with no theoretical background to justify these choices, see e.g. \cite
{PlanckIS} and the references therein.

Our plan is then to consider needlet projectors of the following form:
\begin{equation}
{\Psi }_{j}\left( x,y\right) =\sum_{\ell \geq 0}b_{j}\left( \ell \right)
Z_{\ell ,d}\left( \left\langle x,y\right\rangle \right) \text{ ,}
\label{newfilter}
\end{equation}
where $\left\{ b_{j}\left( \cdot \right) \right\} _{j\in \mathbb{N}}$ is a
sequence of weight functions which generalize the sequence $\left\{ b\left(
\frac{\cdot }{B^{j}}\right) \right\} _{j\in \mathbb{N}}$ characterizing
standard needlets. To make our statements more precise, we will need some
more tools and notation; in particular, we need to introduce a \emph{scale
sequence} $\left\{ S_{j}\right\} _{j\in \mathbb{N}}$, that is, a growing
real-valued sequence such that the support of ${b}_{j}(\cdot )$ is included
in $\Lambda_j=[S_{j-1},S_{j+1}]$ for all $j\in \mathbb{N};$ we are therefore
implicitly maintaining the semi-orthogonality properties of standard
needlets, that is, the support of ${b}_{j}(\cdot )$ and ${b}_{j^{\prime
}}(\cdot )$ are disjoint whenever $\left\vert j-j^{\prime }\right\vert \geq
2.$ For notational simplicity, we shall always assume in the sequel that the
sequence $S_{j}-S_{j-1}$ is increasing, i.e.,
\begin{equation*}
S_{j}-S_{j-1}\leq S_{j+1}-S_{j}\text{ , for all }j\in \mathbb{N};
\end{equation*}
this will allow us to avoid some less elegant statement of results in terms
of the largest between ($S_{j+1}-S_{j})$ and ($S_{j}-S_{j-1})$ - the
substance of the approach is clearly unaltered. The other key ingredient in
the construction is a sequence of kernel functions $\left\{ b_{j}(\cdot)\right\}
_{j\in \mathbb{N}}$ on multipole space, depending on the sequence $\left\{
S_{j}\right\} _{j\in \mathbb{N}},$ for which we require the following
conditions:

\begin{assumption}
\label{ass:bj}
 The sequence of functions $\left\{ b_{j}(\cdot)\right\} $ is such

\begin{enumerate}
\item for all $n,j\in \mathbb{N}$
\begin{equation*}
|D^{(n)}b_{j}(u)|\leq K(n)\frac{1}{(S_{j}-S_{j-1})^{n}}\text{ ,}
\end{equation*}
where the constant $K(n)$ does not depend on $j$;

\item $b_{j}$ has a compact support in $\Lambda_j=\left[ S_{j-1},S_{j+1}\right] $,
with
\begin{equation*}
b_{j}\left( S_{j-1}\right) =b_{j}\left( S_{j+1}\right) =0\text{ , }
b_{j}\left( S_{j}\right) =S_{0}=1\text{ };
\end{equation*}

\item the \emph{partition of unity property} holds, that is,
\begin{equation*}
\sum_{j\geq 0}b_{j}^{2}\left( u\right) =1,\quad \text{for all }u\geq 1\text{
}.
\end{equation*}
\end{enumerate}
\end{assumption}

In the case of standard needlets, the sequence $\left\{ b_{j}(\cdot)\right\}_{j \in \mathbb{N}}$
can be obtained by scaling a function $b(\cdot)$, which is compactly supported
in $\left[B^{-1},B\right]$ for some $B>1$ and with bounded derivatives of any order. In
particular, in the standard construction we have
\begin{equation*}
b_{j}(u):=b\left(\frac{u}{B^{j}}\right)\text{ , }S_{j}:=B^{j}\text{ ,}
\end{equation*}
and hence
\begin{equation*}
|D^{(n)}b_{j}(u)|=\frac{1}{B^{nj}}\left\vert b^{(n)}\left(\frac{u}{B^{j}}
\right)\right\vert \leq \frac{1}{\left(B^{j}-B^{j-1}\right)^{n}}\sup_{u}\left\vert b^{(n)}\left(
\frac{u}{B^{j}}\right)\right\vert \text{ .}
\end{equation*}

The following localization property is the first main result of this paper:

\begin{theorem}[Localization property]
\label{prop:npw} As $j\rightarrow \infty ,$ for all $\theta \in (0,\pi ]$
and $M\in \mathbb{N}$, with $M>d$, there exists a constant $C_{M}>0$ (i.e.,
independent from $x,y$ and $j$ ) such that
\begin{equation}
\left\vert \Psi _{j}\left( \cos \theta \right) \right\vert \leq
C_{M}(S_{j+1}^{d}-S_{j-1}^{d})\max \left\{ \frac{1}{S_{j-1}^{2M}\theta ^{2M}}
,\frac{1}{\left( S_{j}-S_{j-1}\right) ^{2M}\theta ^{2M}}\right\} \text{ .}
\label{npw_loc}
\end{equation}
\end{theorem}

It is important to note that in the standard case (i.e., for $\left\{
S_{j}:=B^{j}\right\} _{j\in \mathbb{N}},$ some $B>1)$ the bound (\ref{npw_loc})
can be written as
\begin{eqnarray*}
\left\vert \Psi _{j}\left( \cos \theta \right) \right\vert &\leq
&C_{M}(B^{(j+1)d}-B^{(j-1)d})\max \left\{ \frac{1}{B^{(j-1)2M}\theta ^{2M}},
\frac{1}{\left( B^{j}-B^{j-1}\right) ^{2M}\theta ^{2M}}\right\} \\
&=&C_{M}(B-\frac{1}{B})B^{jd}\max \left\{ \frac{B^{2M}}{(B^{j}\theta )^{2M}},
\frac{(B/(B-1))^{2M}}{(B^{j}\theta )^{2M}}\right\} \\
&=&C_{M}^{\prime }\frac{B^{jd}}{(B^{j}\theta )^{2M}}\text{ ,}
\end{eqnarray*}
so that Theorem \ref{prop:npw} yields the estimate (\ref{StanLocProp})
which was established in the pioneering papers \cite{npw1,npw2}.

The system of \emph{flexible-bandwidth needlets} $\left\{ \psi _{j,k}\left(
\cdot \right) \right\} _{j,k}~$(or flexible needlets for short) can now be
defined, analogously to (\ref{needproj}), as $\psi _{j,k}\left( \cdot
\right) :\mathbb{S}^{d}\rightarrow \mathbb{R}$ such that
\begin{equation*}
\psi _{j,k}\left( \cdot \right) =\sqrt{\lambda _{j,k}}\sum_{\ell \geq
0}b_{j}\left( \ell \right) Z_{\ell ,d}\left( \left\langle \cdot ,\xi
_{j,k}\right\rangle \right) \text{ ,\ }j\geq 0,\text{ }k=1,\ldots ,K_{j},
\text{\ }
\end{equation*}%
$\left\{ \xi _{j,k},\lambda _{j,k}\right\} _{j\in \mathbb{N},k=1,\ldots,K_{j}}$ representing as before sets
of cubature points and weights such that
\begin{equation*}
\int_{\mathbb{S}^{2}}Y_{\ell, m}(x)\overline{Y}_{\ell ^{\prime },m^{\prime
}}(x)dx=\sum_{k=1}^{K_{j}}Y_{\ell, m}(\xi _{j,k})\overline{Y}_{\ell ^{\prime
},m^{\prime }}(\xi _{j,k})\lambda _{j,k}\text{ , for all }\ell ,\ell ^{\prime
}\leq S_{j+1}\text{ .}
\end{equation*}

We now focus on high-frequency uncorrelation of needlet coefficients; more
precisely, we investigate the corrrelation of the field (\ref{need_coef_cont})
, evaluated by means of (\ref{newfilter}). Assumption (\ref{power1})
requires a form of scale invariance of the angular power spectrum at very
large multipoles/very small scales. In applications, it is often the case
that power spectra may exhibit more complex behaviour, for instance with
sinusoidal oscillations as those which characterize the angular power
spectrum of Cosmic Microwave Background radiation (see \cite{PlanckPS}). In
the present paper, we hence extend and generalize the previous uncorrelation
results (\ref{corr}) considering a much broader class of angular power
spectra for random fields in $\mathbb{S}^{d}$; more precisely, we consider
power spectra taking the form

\begin{assumption}
\label{powspe}
The angular power spectrum satisfies $C_{\ell }=\ell
^{-\alpha }g(\ell )$, where $\alpha >2$, and the function $g:\mathbb{R}
^{+}\rightarrow $ $\mathbb{R}^{+}$ is such that
\begin{equation*}
g_{1}\leq g(u)\leq g_{2},\text{ for some }g_{2}\geq g_{1}>0
\end{equation*}
and for some $\beta \in \lbrack 0,1)$
\begin{equation*}
g^{(r)}(u)=O_{u\rightarrow \infty }(u^{-(1-\beta )r}),\text{ }g^{(r)}(u)=
\frac{d^{r}g(u)}{du^{r}}\text{ .}
\end{equation*}
\end{assumption}

For instance, Assumption \ref{powspe} covers angular power of the form
\begin{equation*}
C_{\ell }=\sum_{p=1}^{P}c_{p}\left\{ d_{p}+\sin (\ell ^{\beta
_{p}}/M_{p})\right\} \ell ^{-\alpha },\text{ }d_{p}>1\text{ , }c_{p},M_{p}>0
\text{ , }0<\beta _{p}<1\text{ for }p=1,...,P,\text{ }
\end{equation*}
thus exhibiting much richer oscillations than allowed in (\ref{power1}).

We now investigate uncorrelation properties in this broader framework, and
we establish our second main result.

\begin{theorem}[Uncorrelation property]
\label{prop:unco} Under Assumptions \ref{ass:bj} and \ref{powspe}, there
exists positive constants $C_{N}$ (depending on $\alpha ,d$ and $g(\cdot )),$
such that, as $j\rightarrow \infty $ we have
\begin{equation}
\mathrm{Corr}(\beta _{j}(x),\beta _{j}(y))
\leq C_{N}\times \max \left\{ \frac{1}{(S_{j-1}^{(1-\beta )}\theta )^{2N}},
\frac{1}{((S_{j}-S_{j-1})\theta )^{2N}}\right\} \text{ }.  \label{Corr3}
\end{equation}
\end{theorem}

As we discuss in the subsection below, this result generalize uncorrelation
properties in the literature even in the standard needlet case $S_{j}=B^{j}$
for $j=1,2,...,$ and hence we believe it can have considerable importance
for applications.

\subsection{Discussion}

Some remarks are in order:

\begin{itemize}
\item Given the localization result established in Theorem \ref{prop:npw},
and the details of the construction of the needlet kernel, it can be easily
verified that flexible needlets form a \emph{tight frame} and they
allow for exact reconstruction formulae. More formally, for all $f\in L^{2}(
\mathbb{S}^{d})$ it is standard to show that the corresponding needlet
coefficients satisfy
\begin{eqnarray*}
\sum_{j\in \mathbb{N}}\sum_{k=1}^{K_j}\beta _{j,k}^{2} &=&\sum_{j\in \mathbb{N}}\sum_{k=1}^{K_j}\lambda _{j,k}\left[ \sum_{\ell \in
\Lambda_j}b_{j}(\ell )a_{\ell, m}Y_{\ell ,m}(\xi _{j,k})\right]
^{2} \\
&=&\sum_{j\in\mathbb{N}}\sum_{\ell _{1},\ell _{2}\in
\Lambda_j}\sum_{m_{1},m_{2}=1}^{N_{\ell;d}}b_{j}(\ell _{1})b_{j}(\ell _{2})a_{\ell
_{1},m_{1}}\overline{a}_{\ell _{2},m_{2}}\\ &&\times\sum_{k=1}^{K_j}\lambda _{j,k}Y_{\ell _{1},m_{1}}(\xi _{j,k})
\overline{Y}_{\ell _{2},m_{2}}(\xi _{j,k}) \\
&=&\sum_{j\in\mathbb{N}} \sum_{\ell _{1},\ell _{2}\in
\Lambda_j}\sum_{m_{1},m_{2}=1}^{N_{\ell;d}}b_{j}(\ell _{1})b_{j}(\ell _{2})a_{\ell
_{1},m_{1}}\overline{a}_{\ell _{2},m_{2}}\delta _{\ell _{1}}^{\ell _{2}}\delta
_{m_{1}}^{m_{2}} \\
&=&\sum_{\ell \in \mathbb{N}}\sum_{m=1}^{N_{\ell;d}}\sum_{j\in \mathbb{N}}b_{j}^{2}(\ell )\left\vert a_{\ell ,m}\right\vert
^{2}=\sum_{\ell \in \mathbb{N}}\sum_{m=1}^{N_{\ell;d}}\left\vert a_{\ell, m}\right\vert ^{2}=\left\Vert
f\right\Vert _{L^{2}(\mathbb{S}^{d})}^{2}\text{ .}
\end{eqnarray*}
Likewise, it can be shown that the following reconstruction formula holds:
\begin{equation*}
f(\cdot )=\sum_{j\in \mathbb{N}}\sum_{k=1}^{K_j}\beta _{j,k}\psi _{j,k}(\cdot )\text{ in }L^{2}(\mathbb{S}
^{d})\text{ .}
\end{equation*}
The details here are identical to those in the seminal papers by \cite{npw1, npw2}, and are therefore omitted for brevity's sake.

\item In the standard needlet case and under \eqref{power1}, \eqref{Corr3}
leads to the following bound:
\begin{equation*}
C_{N}\times \max \left\{ \frac{1}{(B^{(j-1)(1-\beta )}\theta )^{2N}},\frac{
1}{((B^{j}-B^{j-1})\theta )^{2N}}\right\}
\leq \frac{C_{N}}{(B^{(j-1)(1-\beta )}\theta )^{2N}}\text{ .}
\end{equation*}
As mentioned above, this result generalizes to all $\beta \in \lbrack 0,1)$
the uncorrelation bound for needlet coefficients which was given for $\beta
=0$ by \cite{bkmpAoS} and then exploited in a number of subsequent papers to
construct statistical procedures with an asymptotic justification, in the
high-frequency sense.

\item In the general case, asymptotic uncorrelation can continue to hold for
$\beta >0$, but the upper bound is less and less efficient as $\beta $
grows; indeed assuming that $S_{j-1}^{1-\beta }<(S_{j}-S_{j-1})$, for all $
\theta \in \lbrack 0,\pi ]$ we get
\begin{equation*}
\mathrm{Corr}(\beta _{j}(x),\beta _{j}(y))\leq \max \left\{ 1,\frac{C_{N}}{
(S_{j-1}^{1-\beta }\theta )^{2N}}\right\} \text{ .}
\end{equation*}
It should be noted that in the discretized case the construction of cubature
points is such that their minimum distance decays as
\begin{equation*}
d_{j}:=\min_{k,k^{\prime }}d_{\mathbb{S}^{2}}(\xi _{jk},\xi _{jk^{\prime }})\simeq
S_{j-1}^{-1}\text{ },\text{ as }j\rightarrow \infty \text{ .}
\end{equation*}
For $\beta =0,$ it then follows that needlet coefficients have correlations
decaying to zero (as $j\rightarrow \infty $) when evaluated on any pair of
locations whose distance decays more slowly than $d_{j}.$ This is no longer
the case for less regular power spectra: indeed for $\beta >0$ to ensure
asymptotic uncorrelation we must consider pair of coefficients whose
distance is fixed or decays to zero more slowly than $d_{j}^{1-\beta }$ .
\end{itemize}

\subsection{Some Simple Applications}

The uncorrelation properties of spherical needlets have allowed for an
enormous amount of applications in statistical inference in the last few
years, among which we mention subsampling techniques (\cite{BKMPBer}),
Whittle estimation of the model parameters (\cite{DLM13}), point source
detection (\cite{Cheng}), testing for isotropy (\cite{PlanckIS}), and many
others. For brevity's sake we do not develop these applications in the
broader framework considered in this paper; we just include a simple examples on goodness of fit testing.

As it is often the case in the analysis (for
instance) of CMB data, we assume a Gaussian isotropic random field $\left\{
f(.)\right\} $ is observed on a region $D\subset \mathbb{S}^{2}$, and out of
the observations in this region we need to check goodness of fit for some
given model for the angular power spectrum, $\left\{ C_{\ell }=C_{\ell
}(\theta )\right\} _{\ell \in \mathbb{N}}.$ For any $j \in \mathbb{N}$, let $\Xi_j$ denote the grid of  cubature points $\left\{\xi_{j,k}\right\}_{k=1,\ldots,K_j}$. Consider the following testing
procedure:

a) take a needlet construction such that $S_{j-1}^{1-\beta }<(S_{j}-S_{j-1})$, for $j=1,2,\ldots,$ so that we impose a lower bound on the width of $
(S_{j}-S_{j-1})$ (i.e., $(S_{j}-S_{j-1})/S_{j-1}$ can shrink to zero, but $\{(S_{j}-S_{j-1})/S_{j-1}^{1-\beta }\}$ cannot); compute the needlet
coefficients $\left\{ \beta _{j,k}\right\} _{k=1,\ldots,K_j}$

b) choose a subset $D_j$ of these coefficients such that, for $k \in D_j$, $\xi_{j,k} \in \Xi_{j}\cap D,$
and for all $k,k^{\prime } \in D_{j}$ one has $d_{\mathbb{S}^{2}}(\xi
_{j,k},\xi _{j,k^{\prime }})>\delta /S_{j-1}^{1-\beta -\varepsilon },$ for $\delta
,\varepsilon >0,$ and at the same time $\mathrm{card}\left\{ D_{j}\right\}
\rightarrow \infty $ as $j\rightarrow \infty $ (the elements of $D_{j}$ can
be viewed as a subsampling of the cubature points in the grid $\Xi _{j}$ with some constraints on their distance).
Note that for all $M>0,$ there exist $C_{M}$ such that
\begin{equation*}
\mathrm{Corr}(\beta _{j,k},\beta _{j,k^{\prime }})\leq \frac{C_{M}}{(S_{j-1}^{1-\beta
}d_{\mathbb{S}^{2}}(\xi _{j,k},\xi _{j,k^{\prime }}))^{M}}\leq C_{M}\delta
^{-M}S_{j-1}^{-M\varepsilon }\text{ for all }j\in \mathbb{N},
\end{equation*}
hence in particular $\mathrm{Corr}(\beta _{j,k},\beta _{j,k^{\prime }})\rightarrow 0$
as $j\rightarrow \infty $ for all $k,k^{\prime }\in D_{j}$ .

c) now compute
\begin{eqnarray*}
I_{j} =\frac{1}{\sqrt{2\mathrm{card}\left\{ D_{j}\right\} }}\sum_{k\in
D_{j}}\left\{ \frac{\beta _{j,k}^{2}}{\mathbb{E}\left\{ \beta
_{j,k}^{2}\right\} }-1\right\} \text{ , where }
\mathbb{E}\left\{ \beta _{j,k}^{2}\right\} =\sum_{\ell \in
\Lambda_j}b_{j}^{2}(\ell )\frac{2\ell +1}{4\pi }C_{\ell }\text{ .}
\end{eqnarray*}
It is immediate to see that $\mathbb{E}\left\{ I_{j}\right\} =0$ and
\begin{eqnarray*}
\mathrm{Var}\left\{ I_{j}\right\} &=&\frac{1}{2\mathrm{card}\left\{ D_{j}\right\} }\sum_{k\in
D_{j}}\mathrm{Var}\left\{ \frac{\beta _{j,k}^{2}}{\mathbb{E}\left\{ \beta
_{j,k}^{2}\right\} }\right\} \\
&&+\frac{1}{2\mathrm{card}\left\{ D_{j}\right\} }\sum_{k,k^{\prime }\in D_{j},k\neq
k^{\prime }}\mathrm{Cov}\left\{ \frac{\beta _{j,k}^{2}}{\mathbb{E}\left\{ \beta
_{j,k}^{2}\right\} },\frac{\beta _{j,k^{\prime }}^{2}}{\mathbb{E}\left\{ \beta
_{j,k^{\prime }}^{2}\right\} }\right\} \\
&=&1+A_{j}
\end{eqnarray*}
where, using the Diagram (Wick's) Formula (see \cite{NouPebook}, p. 202)
\begin{eqnarray*}
A_{j} &=&\frac{1}{2\mathrm{card}\left\{ D_{j}\right\} }\sum_{k,k^{\prime }\in
D_{j},k\neq k^{\prime }}\mathrm{Cov}\left\{ \frac{\beta _{j,k}^{2}}{\mathbb{E}\left\{
\beta _{j,k}^{2}\right\} },\frac{\beta _{j,k^{\prime }}^{2}}{\mathbb{E}\left\{
\beta _{j,k^{\prime }}^{2}\right\} }\right\} \\
&=&\frac{1}{\mathrm{card}\left\{ D_{j}\right\} }\sum_{k,k^{\prime }\in D_{j},k\neq
k^{\prime }}\mathrm{Corr}^{2}\left\{ \frac{\beta _{j,k}}{\sqrt{\mathbb{E}\left\{ \beta
_{j,k}^{2}\right\} }},\frac{\beta _{j,k^{\prime }}}{\sqrt{\mathbb{E}\left\{
\beta _{j,k^{\prime }}^{2}\right\} }}\right\} \\
&\leq &\mathrm{card}\left\{ D_{j}\right\} \times C_{M}^{2}\delta
^{-2M}S_{j-1}^{-2M\varepsilon }\rightarrow 0\text{ , as }j\rightarrow \infty
\text{ ,}
\end{eqnarray*}
by recalling $\mathrm{card}\left\{ D_{j}\right\} =O(S_{j})$ and choosing $M$ such
that $S_{j}=o(S_{j-1}^{2M\varepsilon })$. We have thus shown that $
\lim_{j\rightarrow \infty }\mathrm{Var}\left( I_{j}\right)=1$ .

d) finally, it is now a standard computation to show that
\begin{eqnarray*}
\mathrm{Cum}_{4}\left\{ I_{j}\right\} &=&\frac{1}{\left\{ 2\mathrm{card}\left\{ D_{j}\right\}
\right\} ^{2}}\mathrm{Cum}_{4}\left\{ \sum_{k\in D_{j}}\left\{ \frac{\beta _{j,k}^{2}}{
\mathbb{E}\left\{ \beta _{j,k}^{2}\right\} }-1\right\} \right\} \\
&=&\frac{1}{\left\{ 2\mathrm{card}\left\{ D_{j}\right\} \right\} ^{2}}
\end{eqnarray*}
\begin{equation*}
\times O\left\{ \sum_{k_{1},k_{2},k_{3},k_{4}\in D_{j}}\mathrm{Corr}^{2}\left\{ \frac{
\beta _{j,k_{1}}^{2}}{\mathbb{E}\left\{ \beta _{j,k_{1}}^{2}\right\} },\frac{
\beta _{j,k_{2}}^{2}}{\mathbb{E}\left\{ \beta _{j,k_{3}}^{2}\right\} }\right\}
...\mathrm{Corr}^{2}\left\{ \frac{\beta _{j,k_{4}}^{2}}{\mathbb{E}\left\{ \beta
_{j,k_{4}}^{2}\right\} },\frac{\beta _{j,k_{2}}^{2}}{\mathbb{E}\left\{ \beta
_{j,k_{1}}^{2}\right\} }\right\} \right\}
\end{equation*}

\begin{equation*}
=O\left( \frac{1}{\mathrm{card}\left\{ D_{j}\right\} }\right) \text{ .}
\end{equation*}
It is then an immediate application of the Malliavin-Stein method (see \cite
{NouPebook} and the references therein) to prove that a (quantitative)
Central Limit Theorem holds for the sequence $\left\{ I_{j}\right\}_{j \in \mathbb{N}} ,$ thus
making well-principled goodness of fit tests available.

In a similar manner, under these broader circumstances extensions can be
implemented for needlet based-procedures in a number of areas of theoretical
and applied interest: we mention for instance high-frequency maximum
likelihood estimates (as investigated by \cite{DLM13} in the standard
needlet case), polyspectra estimation (see e.g., \cite{Cammarota}), isotropy
testing (see \cite{PlanckIS}), power spectrum estimation (see \cite{BKMPBer,PlanckPS}), point source detection (\cite{carron,Cheng}) and
many others. For brevity's sake, we do not discuss the implementation
details here.

\subsection{Plan of the Paper}

The properties of flexible-bandwidth needlets in terms of localization in
real space are discussed in Section \ref{Localization}, while uncorrelation
properties are investigated in Section \ref{Uncorrelation}; an explicit
construction for the sequence of kernels $\left\{ b_{j}(\cdot )\right\}
_{j\in \mathbb{N}}$ is given in the Appendix (Section \ref{Appendix}).

\subsection{Acknowledgments}
The research by CD was supported by \lq\lq Bando di Ateneo Sapienza" RM120172B7A31FFA- Costruzione di basi multiscala e trasformate wavelet per applicazioni in ambito numerico e statistico. The research by DM was partially
supported by the MIUR Departments of Excellence Program Math@Tov, CUP
E83C18000100006. The research by APT was partially supported by Progetto di
Eccellenza, Dipartimento di Scienze Matematiche, Politecnico di Torino, CUP:
E11G18000350001. APT and CD have
been also partially supported by the German Research Foundation (DFG) via RTG 2131.

\section{Localization Properties \label{Localization}}

In this section we will establish a localization property which generalizes
analogous results for standard needlets in \cite{npw1}, Mexican needlets in
\cite{dur,gm2} and scale-directional wavelets in \cite{mcewen}.

Let us first recall some useful notation. Consider a real-valued sequence $
\left\{ r_{\ell }: \ell \geq 0\right\} $ and let the \emph{discrete
difference operators} $\Delta ^{+}$, $\Delta ^{-}$ be defined by
\begin{equation*}
\Delta ^{+}r_{\ell }=r_{\ell +1}-r_{\ell }\text{ , \ }\Delta ^{-}r_{\ell
}=r_{\ell }-r_{\ell -1}.
\end{equation*}
These operators can be viewed as discrete versions of derivation on
sequences (see also \cite[Definition 2.1]{mayeli}), and can be used to
define
\begin{equation*}
\Upsilon _{d}\left( \ell \right) =\upsilon _{1;d}\left( \ell \right) \Delta
^{-}\Delta ^{+}+\upsilon _{0;d}\left( \ell \right) \Delta ^{+},\text{ }d\geq
2\text{ ,}
\end{equation*}
where
\begin{eqnarray*}
\upsilon _{1,d}\left( \ell \right) &:=&\frac{\ell }{2\left( \ell +\eta
_{d}\right) }=\frac{\ell }{2\ell +d-1}=\frac{1}{2}-\frac{d-1}{4\ell +2d-2}
\text{ },\quad \\
\upsilon _{0,d}\left( \ell \right) &:=&\frac{2\eta _{d}}{2\left( \ell +\eta
_{d}\right) }=\frac{d-1}{2\ell +d-1}\leq \frac{d-1}{2\ell }\text{ }.
\end{eqnarray*}
Our main result is the following.

\begin{proposition}[Localization]
\label{lemma:mayeli2} Let $\Psi _{j}(\cdot)$ be defined as
\begin{equation*}
\Psi _{j}\left( \cos \theta \right) :=\sum_{\ell \in \Lambda_j} b_{j}\left( \ell \right) Z_{\ell ;d}\left( \cos \theta \right) \text{ },
\text{ }j\in \mathbb{N},
\end{equation*}
where for all $M>0$ there exists a positive constant $C_{M}>0$ such that
\begin{equation}
\left( \Delta ^{-}\right) ^{M}\left( \Delta ^{+}\right) ^{M}b_{j}\left( \ell
\right) \leq C_{M}\frac{1}{\left( S_{j}-S_{j-1}\right) ^{2M}}.
\label{eq:deltaplus}
\end{equation}
Then, it holds that
\begin{equation*}
\left\vert (\cos \theta -1)^{M}\Psi _{j}\left( \cos \theta \right)
\right\vert \leq C_{M}\left( S_{j+1}^{d}-S_{j-1}^{d}\right) \max \left\{
\frac{1}{S_{j-1}^{2M}},\frac{1}{\left( S_{j}-S_{j-1}\right) ^{2M}}\right\}
\end{equation*}
and hence, because $\theta ^{2}=O(\left\vert \cos \theta -1\right\vert )$
for $\theta \in (0,\pi)$

\begin{eqnarray*}
\left\vert \Psi _{j}\left( \cos \theta \right) \right\vert &\leq&
C_{M}(S_{j+1}^{d}-S_{j-1}^{d})\max \left\{ \frac{1}{S_{j-1}^{2M}\theta
^{2M}},\frac{1}{\left( S_{j}-S_{j-1}\right) ^{2M}\theta ^{2M}}\right\}.\notag
\end{eqnarray*}
\end{proposition}

The proof of the previous results requires the following two lemmas, which
are generalizations to $\mathbb{S}^{d}$ of \cite[Lemma 4.1]{mayeli}, where $
\mathbb{S}^{2}$ was considered.

\begin{lemma}
\label{lemma:mayeli1}Let
\begin{equation*}
q\left( \cos \theta \right) :=\sum_{\ell \geq 0}r_{\ell }\frac{\left( \ell
+\eta _{d}\right) }{\eta _{d}\omega _{d}}G_{\ell }^{\left( \eta _{d}\right)
}\left( \cos \theta \right) =\sum_{\ell \geq 0}r_{\ell }Z_{\ell ;d}\left(
\cos \theta \right) , 
\end{equation*}
where $\left\{ r_{\ell}: \ell \geq 0 \right\}$  is a real-valued sequence.
Then, for any $N\in \mathbb{N}$,
\begin{equation}
\left( \cos \theta -1\right) ^{N}q\left( \cos \theta \right) =\sum_{\ell
\in \mathbb{N}}r_{\ell ;d}^{\left( N\right) }Z_{\ell ;d}\left( \cos \theta \right) ,
\label{eq:lemma1}
\end{equation}
where $r_{\ell ;d}^{\left( N\right) }:=\Upsilon _{d}^{N}\left( \ell \right)
r_{\ell }.$
\end{lemma}

\begin{proof}[Proof of Lemma \protect\ref{lemma:mayeli1}]
	Recall first the identity, valid for $x\in \lbrack -1,1],$ $\ell \in \mathbb{
		N}_{0}$
	\begin{eqnarray*}
		&&(x-1)\left[ 2(\ell +\eta _{d})G_{\ell }^{\left( \eta _{d}\right) }\left(
		x\right) \right] \\
		&&\quad \quad =\left( \ell +1\right) G_{\ell +1}^{\left( \eta _{d}\right)
		}\left( x\right) -2\left( \ell +2\eta _{d}\right) G_{\ell }^{\left( \eta
			_{d}\right) }\left( x\right) +\left( \ell +2\eta _{d}-1\right) G_{\ell
			-1}^{\left( \eta _{d}\right) }\left( x\right) ,
	\end{eqnarray*}
see \cite[Equation 22.7.3]{handbook}. With the convention $G_{-1}^{\left(
		\eta _{d}\right) }\left( x\right) =0$ for any $x\in \left[ -1,1\right] $, $
	r_{-1}=0$, and writing $Z_{\ell }(\cos \theta )=2(\ell +\eta _{d})G_{\ell
	}^{\left( \eta _{d}\right) }\left( \cos \theta \right) $, we have
	\begin{align*}
	\sum_{\ell \geq 0}& r_{\ell }\left[ \left( x-1\right) Z_{\ell }\left(
	x\right) \right] \\
	& =\sum_{\ell \geq 0}r_{\ell }\left[ \frac{\ell +1}{2\left( \left( \ell
		+1\right) +\eta _{d}\right) }Z_{\ell +1}\left( x\right) -Z_{\ell }\left(
	x\right) +\frac{\ell +2\eta _{d}-1}{2\left( \left( \ell -1\right) +\eta
		_{d}\right) }Z_{\ell -1}\left( x\right) \right] \\
	& =\sum_{\ell \geq 1}r_{\ell -1}\frac{\ell }{2\left( \ell +\eta _{d}\right) }
	Z_{\ell }(x)-\sum_{\ell \geq 0}r_{\ell }Z_{\ell }(x)+\sum_{\ell \geq
		-1}r_{\ell +1}\frac{\ell +2\eta _{d}}{2\left( \ell +\eta _{d}\right) }
	Z_{\ell }\left( x\right) \\
	& =\sum_{\ell \geq 0}\left[ \frac{\ell }{2\left( \ell +\eta _{d}\right) }
	r_{\ell -1}-\frac{2\left( \ell +\eta _{d}\right) }{2\left( \ell +\eta
		_{d}\right) }r_{\ell }+\frac{\ell +2\eta _{d}}{2\left( \ell +\eta
		_{d}\right) }r_{\ell +1}\right] Z_{\ell }\left( x\right) \\
	& =\sum_{\ell \geq 0}\left[ \frac{\ell }{2\left( \ell +\eta _{d}\right) }
	\left( r_{\ell -1}-2r_{\ell }+r_{\ell +1}\right) +\frac{2\eta _{d}}{2\left(
		\ell +\eta _{d}\right) }\left( r_{\ell +1}-r_{\ell }\right) \right] Z_{\ell
	}\left( x\right) \\
	& =\sum_{\ell \geq 0}r_{\ell }^{\left( 1\right) }Z_{\ell }\left( x\right) .
	\end{align*}
	Now, fixing $x=\cos \theta $ and dividing by $2\eta _{d}\omega _{d}$, we
	obtain that
	\begin{equation*}
	\left( \cos \theta -1\right) q\left( \cos \theta \right) =\sum_{\ell \geq
		0}r_{\ell }^{\left( 1\right) }Z_{\ell ,d}\left( \cos \theta \right) .
	\end{equation*}
	Iterating, we obtain \eqref{eq:lemma1}.
	
	\end{proof}

Lemma \ref{lemma:mayeli1} exploits the natural fact that if a function $
q\left( u\right) $ can be expanded into Gegenbauer polynomials with
coefficients $\left\{ r_{\ell }: \ell \geq 0 \right\} $, then also $\left(
u-1\right) ^{N}q\left( u\right) $ can also be expanded with coefficients
which can explicitly computed by properly applying iteratively the
difference operators to the sequence $\left\{ r_{\ell }: \ell \geq 0 \right\} $. In some
sense, this can be viewed as an extension to the spherical domain of the
classical duality relationships between Fourier transforms and derivatives.

Let us prove now that $b_j(\ell)$ satisfies (\ref{eq:deltaplus}).

\begin{lemma}
\label{UpsilonN} For any $N\in \mathbb{N}$
\begin{equation*}
\Upsilon _{d}^{N}\left( \ell \right) b_{j}(\ell )\leq \frac{1}{2^{N}}
(2N)!\max_{u}\{D^{(2N)}b_{j}(u)\}+\sum_{i=0}^{2N-1}\frac{C(i)}{\ell ^{2N-i}}
\max_{u}\{D^{(i)}b_{j}(u)\}.
\end{equation*}
\end{lemma}

\begin{proof}
Let us consider first $N=1$. Then we have
\begin{eqnarray*}\label{eq}
\Upsilon _{d}\left( \ell \right) b_j(\ell)&=&\left(\upsilon _{1;d}\left( \ell \right) \Delta
^{-}\Delta ^{+}+\upsilon _{0;d}\left( \ell \right) \Delta ^{+} \right) b_j(\ell)\\&
=&\upsilon _{1;d}\left( \ell \right) \Delta
^{-}\left( b_{j}(\ell+1)-b_j(\ell)\right)+\upsilon _{0;d}\left( \ell \right) \left(b_j(\ell+1)-b_j(\ell)\right)\\&
=&\frac{\ell}{2(\ell+\eta_d)}\left( b_{j}(\ell+1)-b_j(\ell)-(b_j(\ell)-b_j(\ell-1))\right)+\frac{2\eta_d}{2(\ell+\eta_d)} \left(b_j(\ell+1)-b_j(\ell)\right).
\end{eqnarray*}

The Mean Value Theorem implies that there exists $u_1 \in (\ell, \ell+1)$ and $u_2 \in (\ell-1,\ell)$ such that
\begin{equation}
\Upsilon _{d}\left( \ell \right) b_j(\ell)=\frac{\ell}{2(\ell+\eta_d)}\left( b_{j}'(u_1)-b_j'(u_2)\right)+\frac{2\eta_d}{2(\ell+\eta_d)} b_j'(u_1).
\end{equation}
Applying once more the Mean Value Theorem we have that there exists $a_1 \in (u_1,u_2)$ such that
\begin{equation*}
\Upsilon _{d}\left( \ell \right) b_j(\ell)
=\frac{\ell}{2(\ell+\eta_d)}\left( b_{j}''(a_1)(u_1-u_2)\right)+\frac{2\eta_d}{2(\ell+\eta_d)} b_j'(u_1).
\end{equation*}
Hence
\begin{equation*}
\Upsilon _{d}\left( \ell \right) b_j(\ell)
\leq \frac{2\ell}{2(\ell+\eta_d)} \max \left|b_{j}''(u)\right| +\frac{2\eta_d}{2(\ell+\eta_d)} \max \left|b_j'(u)\right|.
\end{equation*}
Our assumptions on $b_j(\ell)$ and its derivatives allow to complete the proof for $N=1$. The general case follows applying $\Upsilon_{d}^{N-1}$ on (\ref{eq}) and using induction, for $N \in \mathbb{N}$.

\begin{remark}
	Observe that
	$$\frac{2\ell}{2(\ell+\eta_d)} \max \left|b_{j}''(u)\right| \leq \frac{2\ell}{2(\ell+\eta_d)} \frac{1}{(S_j-S_{j-1})^2} \leq \frac{C_d}{(S_j-S_{j-1})^2} $$
	$$\frac{2\eta_d}{2(\ell+\eta_d)} \max \left|b_j'(u)\right| \leq \frac{2\eta_d}{2(\ell+\eta_d)} \frac{1}{(S_{j}-S_{j-1})} \leq \frac{C^\prime_d}{\ell(S_{j}-S_{j-1})},$$
	where $C_d, C_d^\prime >0$ depend only on $d$.
Then $$\Upsilon _{d}\left( \ell \right) b_j(\ell)\leq
	 C  \max \bigg\{ \frac{1}{S_{j-1}^{2}} , \frac{1}{(S_{j}-S_{j-1})^{2}}\bigg\}. $$
More generally,
	\begin{equation}\label{ups}
	\Upsilon _{d}^N\left( \ell \right) b_j(\ell)
	\leq
	 C(2N)  \max \bigg\{ \frac{1}{S_{j-1}^{2N}} , \frac{1}{(S_{j}-S_{j-1})^{2N}}\bigg\}.
	\end{equation}
\end{remark}

\end{proof}

\begin{remark}
\label{remark} It is immediate to see that, as $j\rightarrow \infty $,
\begin{equation*}
\sum_{\ell \in \Lambda_j}\ell ^{d-1}=\frac{1}{d}
(S_{j+1}^{d}-S_{j-1}^{d})+O(S_{j+1}^{d-1})=O(S_{j+1}^{d}-S_{j-1}^{d})\text{ .
}
\end{equation*}
\end{remark}

\begin{proof}[Proof of Proposition \protect\ref{lemma:mayeli2}]
	For any $j\in \mathbb{N}_{0}$, it suffices to note that
applying Lemma \ref{lemma:mayeli1}
yields
	\begin{eqnarray*}
		\left\vert \left( \cos \theta -1\right) ^{N}\Psi _{j}\left( \cos \theta
		\right) \right\vert &=&\left\vert \sum_{\ell \geq 0}b_j(\ell )^{\left(
			N\right) }Z_{\ell ,d}(\cos \theta )\right\vert .
	\end{eqnarray*}%

Lemma \ref{UpsilonN}, (\ref{ups}) and the conditions on $b_j(\cdot)$  imply that for all $M>0$

\begin{eqnarray*}
&&\left\vert \left( \cos \theta -1\right) ^{M}\Psi _{j}\left( \cos \theta
\right) \right\vert\\&
\leq& C_{M} \max \left\{
\frac{1}{S_{j-1}^{2M}},\frac{1}{\left( S_{j}-S_{j-1}\right) ^{2M}}\right\}
\sum_{\ell \in \Lambda_j}\frac{
		\ell +\eta _{d}}{\eta _{d}\omega _{d}}\left\vert G_{\ell }^{\left( \eta_{d}\right) }(\cos \theta )\right\vert .
 \end{eqnarray*}

In view of Remark \ref{remark}, because $\theta^2=O(\left\vert \cos \theta -1\right\vert)$, we have
	\begin{equation*}
	\left\vert \Psi _{j}\left( \cos \theta \right) \right\vert \leq
	C_{M}^{\prime } \max\bigg\{ \frac{1}{(S_{j-1})^{2M}}, \frac{1}{(S_j-S_{j-1})^{2M}}\bigg\}\frac{(S_{j+1}^d-S_{j-1}^d)}{\theta^{2M}}
	\end{equation*}
	as claimed.
\end{proof}

\section{Uncorrelation Properties\label{Uncorrelation}}

Our last step consists in showing that kernels of the type
\begin{equation*}
\Phi _{j}\left( \cos \theta \right) =\sum_{\ell \in \Lambda_j}b_{j}^{2
}\left( \ell\right) C_\ell Z_{\ell ;d}\left( \cos \theta \right) ,
\end{equation*}
satisfy a localization property under some conditions on the power spectrum $
C_\ell$ specified later. This result will allow us to show that needlet
coefficients are asymptotically uncorrelated for $j \to \infty$.

Recall first that, for all $d=1,2,...$
\begin{eqnarray*}
\left\vert Z_{\ell ;d}\left( \cos \theta \right) \right\vert &\leq &\frac{
2\ell +d-1}{(d-1)}\binom{\ell +d-2}{\ell } \\
&\leq &C_d\times \ell ^{d-1},
\end{eqnarray*}
where the constant $C_d$ depends only on $d$. Now note that
\begin{eqnarray*}
\frac{d^{N}}{du^{N}}(b_{j}(u)^{2}u^{-\alpha }g(u))&=&\sum_{k=0}^{N}\binom{N}{k}\frac{d^{k}}{du^{k}}b_{j}(u)^{2}\frac{d^{N-k}}{
du^{N-k}}(u^{-\alpha }g(u))\\
&=&\sum_{k=0}^{N}\binom{N}{k}\frac{d^{k}}{du^{k}}(a_{j+1}(u)-a_{j}(u))
\sum_{i=0}^{N-k}\frac{d^{i}}{du^{i}}u^{-\alpha }\frac{d^{N-k-i}}{du^{N-k-i}}
g(u)
\\
&=&\sum_{k=0}^{N}\binom{N}{k}\frac{d^{k}}{du^{k}}(a_{j+1}(u)-a_{j}(u))
\sum_{i=0}^{N-k}[-\alpha ]_{i}u^{-\alpha -i}\frac{d^{N-k-i}}{du^{N-k-i}}g(u),
\end{eqnarray*}
where
\begin{equation*}
\lbrack -\alpha ]_{i}:=-\alpha (-\alpha -1)...(-\alpha -i+1)\text{ .}
\end{equation*}
It follows that, for all $\ell $ such that $S_{j-1}\leq \ell \leq S_{j+1},$
we have

\begin{eqnarray*}
\left\vert \frac{d^{N}}{du^{N}}(b_{j}(u)^{2}u^{-\alpha }g(u))\right\vert
_{u=\ell }
 &\leq& C_{N,\alpha }\sum_{k=0}^{N}\binom{N}{k}\frac{1}{(S_{j}-S_{j-1})^{k}}
\sum_{i=0}^{N-k}\ell ^{-\alpha -i}\ell ^{-(N-k-i)(1-\beta )}
\\& \leq& C_{N,\alpha }\ell ^{-\alpha }\ell ^{-N(1-\beta )}\sum_{k=0}^{N}\frac{
\ell ^{k(1-\beta )}}{(S_{j}-S_{j-1})^{k}}\\&=&\sum_{k=0}^{N}\frac{C_{N,\alpha
}\ell ^{-\alpha }}{(S_{j}-S_{j-1})^{k}\ell ^{(N-k)(1-\beta )}}\text{ .}
\end{eqnarray*}

Note that for $(S_{j}-S_{j-1})\geq S_{j-1}^{(1-\beta )}$ the denominator is
bounded below by $S_{j-1}^{N(1-\beta )},$ whereas for $
(S_{j}-S_{j-1})<S_{j-1}^{(1-\beta )}$ we have the smaller bound $
(S_{j}-S_{j-1})^{N}<S_{j-1}^{-N(1-\beta )}$. The bottom line is hence

\begin{equation*}
\left\vert \frac{d^{N}}{du^{N}}(b_{j}(u)^{2}u^{-\alpha }g(u))\right\vert
_{u=\ell }\leq C\times \ell ^{-\alpha }\times \max \left\{ \frac{1}{
S_{j-1}^{N(1-\beta )}},\frac{1}{(S_{j}-S_{j-1})^{N}}\right\} \text{ ,}
\end{equation*}
where $C>0$.

Now consider the correlation function

\begin{equation*}
\Phi (\cos \theta )=\sum_{\ell \in \Lambda_j}b_{j}(\ell )^{2}\ell
^{-\alpha }g(\ell )Z_{\ell ,d}(\cos \theta )\text{ ;}
\end{equation*}
we have the bound

\begin{eqnarray*}
|\cos \theta -1|^{N}\Phi (\cos \theta )&=&\sum_{\ell \in (S_{j-1},S_{j+1})}\left\{
\Upsilon _{d}^{N}\left( \ell \right) b_{j}(\ell )^{2}\ell ^{-\alpha }g(\ell
)\right\} Z_{\ell ,d}(\cos \theta )\\&
\leq& C\times \max \left\{ \frac{1}{S_{j-1}^{2N(1-\beta )}},\frac{1}{
(S_{j}-S_{j-1})^{2N}}\right\} \sum_{\ell \in (S_{j-1},S_{j+1})}\ell ^{-\alpha
}Z_{\ell ,d}(\cos \theta )\\&
\leq& C\times \max \left\{ \frac{1}{S_{j-1}^{2N(1-\beta )}},\frac{1}{
(S_{j}-S_{j-1})^{2N}}\right\} \\&
\times& \min \left\{ (S_{j+1}-S_{j-1})S_{j-1}^{d-\alpha -1},S_{j-1}^{d-\alpha
}\right\} \text{ ,}
\end{eqnarray*}
where $C>0$. It is easy to check that the denominator (i.e., the variance of the field $
\beta _{j}(\cdot)$) is given by
\begin{equation*}
\sum_{\ell \in\Lambda_j}b_{j}(\ell )^{2}\ell ^{-\alpha }g(\ell )\frac{
\ell +\eta _{d}}{\eta _{d}\omega _{d}}G_{\ell }^{(\eta _{d})}(1).
\end{equation*}
Because $b_{j}^{(1)}\leq K/(S_{j}-S_{j-1})$ and $b_{j}(\ell )=1$ for some $
\ell \in \Lambda_j,$ by a simple first-order Taylor expansion it is
readily seen that there exist $S_{j-1}^{\prime },S_{j+1}^{\prime }$ which
satisfy the following conditions:

\begin{eqnarray*}
S_{j-1} &<&S_{j-1}^{\prime }<S_{j+1}^{\prime }<S_{j+1}\text{ ,} \\
(S_{j+1}^{\prime }-S_{j-1}^{\prime }) &>&c_{1}(S_{j+1}-S_{j-1})\text{ , some
}c_{1}>0\text{ ,} \\
b_{j}(\ell ) &>&c_{2}>0\text{ for all }\ell \in (S_{j-1}^{\prime
},S_{j+1}^{\prime })\text{ ,}
\end{eqnarray*}
where the constants $c_{1},c_{2}$ are absolute (they do not depend on $j$).
Hence we have the lower bound
\begin{eqnarray*}
\sum_{\ell \in \Lambda_j}b_{j}(\ell )^{2}\ell ^{-\alpha }g(\ell )\frac{
\ell +\eta _{d}}{\eta _{d}\omega _{d}}G_{\ell }^{(\eta _{d})}(1)
&\geq &c_{2}^{2}\sum_{\ell \in (S_{j-1}^{\prime },S_{j+1}^{\prime })}\ell
^{-\alpha }g(\ell )\frac{\ell +\eta _{d}}{\eta _{d}\omega _{d}}G_{\ell
}^{(\eta _{d})} \\
&\geq & C \times \min \left\{ (S_{j+1}-S_{j-1})S_{j-1}^{d-\alpha
-1},S_{j-1}^{d-\alpha }\right\} \text{ ,}
\end{eqnarray*}
where $C>0$. Then, we have
\begin{eqnarray*}
\mathrm{Corr}(\beta _{j}(x),\beta _{j}(y)) \\
&\leq &C\times \max \left\{ \frac{1}{S_{j-1}^{2N(1-\beta )}},\frac{1}{
(S_{j}-S_{j-1})^{2N}}\right\} \frac{1}{|\cos \theta -1|^{N}} \\
\text{ } &\leq &C^{\prime }\times \max \left\{ \frac{1}{
(S_{j-1}^{(1-\beta )}\theta )^{2N}},\frac{1}{((S_{j}-S_{j-1})\theta )^{2N}}
\right\} \text{,}
\end{eqnarray*}
with $C,C^\prime>0$, as claimed.

\section{Appendix \label{Appendix}: an Explicit Construction for $\left\{
b_{j}(\cdot )\right\} _{j\in \mathbb{N}}$}

\label{sec:construction} In this Appendix, we will provide an explicit
construction of $\left\{ b_{j}(\cdot )\right\} _{j\in \mathbb{N}}$. Most of the steps are a
generalization under the more general circumstances considered in this paper of the procedure which was suggested in \cite{bkmpAoS} for
the standard needlet case.

Let us define a sequence of functions $a_{j}:\mathbb{\ R^{+}}\rightarrow
\left[ 0,1\right] $ such that
\begin{equation*}
a_{j}\in C^{\infty }\left( \mathbb{R^{+}}\right) ,\text{ }a_{j}\left(
u\right) =1\text{ for }\left\vert u\right\vert \leq S_{j-1}\text{ for }j\geq
1,
\end{equation*}
(so that $a_{0}\left( 0\right) =1)$, and
\begin{equation*}
0<a_{j}\left( u\right) \leq 1\text{ for }u\in \left[ S_{j-1},S_{j}\right] \text{ }.
\end{equation*}

We introduce now a sequence of window functions $\left\{ b_{j}:j\in \mathbb{N
}\right\} $ given by
\begin{equation}
b_{j}\left( u\right) :=\sqrt{a_{j+1}\left( u\right) -a_{j}\left( u\right) }.
\label{eq:bjfun}
\end{equation}

Observe that
\begin{equation}  \label{eq:bjexplicit}
b_{j}(u)=
\begin{cases}
\sqrt{1-a_j(u)} & S_{j-1}< u\leq S_{j} \\
\sqrt{a_{j+1}(u)} & S_{j}< u< S_{j+1} \\
0 & \text{otherwise}
\end{cases}
.
\end{equation}

\begin{lemma}
For any $j \in \mathbb{N}$, it holds that $b_j\in C^\infty$.
\end{lemma}

\begin{proof}
	
	For any $j \in \mathbb{N}$, it follows from Equation \eqref{eq:bjexplicit} that $b_j (u) \in C^\infty$ in
	$	(0, S_{j-1}) \cup (S_{j-1}, S_{j + 1})\cup (S_{j+1},\infty)$. To establish the smoothness of $b_j(\cdot)$ we need to study the behaviour of $a_j (u) $ (and, consequently, $b_j (u)$) in $ u = S_{j-1}, S_{j+1}.$ In order to do so we prove that left and right derivatives coincide in these two points. Let us start by proving that $b_j(\cdot)$ is $C^\infty$ in $S_{j+1}$.
	
	The Taylor series of $a_{j+1}$ centered at $S_{j+1}$ can be written as $$a_{j+1}(u)= a_{j+1}(S_{j+1})+\dots+ \frac{a_{j+1}^{(n)}{(S_{j+1})}}{n!} (u-S_{j+1})^n+o((u-S_{j+1})^n) \mbox{ as } u\to S_{j+1}$$
	for all $n$. Since $a_j(u)$  $ \in C^{\infty}$ and  $a_{j+1}(S_{j+1}^+)^{(k)}=0$ we get that $a_{j+1}(S_{j+1})^{(k)}=0$ for all $k=0,\dots, n$ and then
	
	$$a_{j+1}(u)= o((u-S_{j+1})^n)$$ for all $n$, as $u \to S_{j+1}$.
	
	Moreover,  $a_{j}(S_{j+1}^-)=0$ and then $b_j(u)=\sqrt{a_{j+1}(u)}$. Hence we get that $$\frac{b_j(u)-b_j(S_{j+1})}{u-S_{j+1}}=\dfrac{\sqrt{a_{j+1}(u)}-0}{u-S_{j+1}}= \dfrac{o(u-S_{j+1})}{u-S_{j+1}}=o(1)$$
and then $b_j \in C^1$ in $S_{j+1}$.
	
	A similar argument can be implemented for  $u=S_{j-1}$. Indeed, we note that $a_{j}(S_{j-1})=1$ and since $a_j(u)$ is $C^{\infty}$ and it is zero on $S_{j-1}^-$, we have that $a_{j}(S_{j-1})^{(k)}=0$ for all $k=1,\dots, n$. Then a Taylor series expansion leads to
	
	$$a_{j}(u)= 1+o((u-S_{j+1})^n)$$
for all $n$. Moreover, since $a_j$ is continuous and it is equal to  1 in $S_{j-1}^-$ we have that $a_{j}(S_{j-1}^+)=1$ and also
	$ a_{j+1}(S_{j-1})=1$. Hence in a neighborhood of $S_{j-1}$ we have that $b_j(u)=\sqrt{1-a_j(u)}$ so that the quotient derivative of $b_j(\cdot)$ from the right is
	$$\dfrac{\sqrt{1-a_j(u)}-0}{{u-S_{j-1}}}=\dfrac{o(u-S_{j-1})}{u-S_{j-1}}=o(1).$$ Then $b_j \in C^1$ in $S_{j-1}$ which implies $b_j \in C^1$ in $[0, \infty)$; iterating the procedure proves that $b_j \in C^\infty$.
\end{proof}

We propose here a numerical recipe for $b_{j}(\cdot)$, which is largely analogous
to the proposal developed in \cite{bkmpAoS} for the standard needlet
construction. First introduce the function $\phi \in C_{c}^{\infty }$, given
by
\begin{equation*}
\phi (t)=
\begin{cases}
\exp \left( -\frac{1}{1-t^{2}}\right) & \mbox{ for }t\in \lbrack -1,1] \\
0 & \mbox{ otherwise }
\end{cases}
\end{equation*}
The function $\phi $ belongs to the Schwarz space; consider now
\begin{equation*}
\Phi \left( u\right) =
\begin{cases}
0 & u\leq -1 \\
\frac{\int_{-1}^{u}\phi \left( t\right) dt}{C_{\Phi }} & u\in (-1,1) \\
1 & u\geq 1
\end{cases}
,
\end{equation*}
where
\begin{equation*}
C_{\Phi }=\int_{-1}^{1}\phi \left( t\right) dt=\int_{-1}^{1}\exp \left( -
\frac{1}{1-t^{2}}\right) dt\simeq 0.444.
\end{equation*}
Also, for any $j\in \mathbb{N}$, define
\begin{equation}
a_{j}(u)=
\begin{cases}
1 & \text{ for }u\in \left[ 0,S_{j-1}\right] \\
\Phi \left( \frac{\left( S_{j}+S_{j-1}-2u\right) }{\left(
S_{j}-S_{j-1}\right) }\right) & \text{ for }u\in \left( \left. S_{j-1},S_{j}
\right] \right. \\
0 & \text{ for }u\in \left[ \left. S_{j},\infty \right) \right.
\end{cases} 
. \label{sabato1}
\end{equation}
Note that in $[S_{j-1},S_{j}]$
\begin{equation*}
a_{j}(u)=\Phi (\tau _{j}(u))
\end{equation*}
where $\tau _{j}$ is a linear transformation defined by
\begin{equation*}
\tau _{j}(u)=m_{j}u+q_{j}
\end{equation*}
with
\begin{equation*}
m_{j}=-\frac{2}{S_{j}-S_{j-1}};\quad q_{j}=\frac{S_{j}+S_{j-1}}{S_{j}-S_{j-1}
}.
\end{equation*}

\begin{remark}
It follows that, for any $r \in \mathbb{N}$,
\begin{eqnarray}
a_{j}^{(r)}(u) &=& \frac{d^r}{du^r}a_j(u)=\tau_j^{(r)}(u)\Phi ^{(r)}\left(
\tau _{j}(u)\right)  \notag \\
&=&\frac{(-2)^{r}}{\left( S_{j}-S_{j-1}\right) ^{r}}\frac{\phi^{(r-1)}\left(
\tau _{j}(u)\right) }{C_\Phi}\text{ .}  \label{eq:ajder}
\end{eqnarray}
\end{remark}

Finally, according to (\ref{eq:bjfun}), we can define a sequence of window functions $\left\{b_j:j\in
\mathbb{N}\right\}$, where $b_j: \mathbb{R}^+\rightarrow \left[0,1\right]$
is such that
\begin{equation}
b_{j}\left( u\right):= \sqrt{a_{j+1}\left( u\right) -a_{j}\left( u\right)}. \label{sabato2}
\end{equation}

\begin{proposition}
\label{bound1} For any $a_j(\cdot)$ defined as in \ref{sabato1} and $n\geq1$
\begin{equation*}
|D^{(n)}a_j(u)|\leq k(n-1)2^{n} \frac{1}{(S_{j}-S_{j-1})^n}
\end{equation*}
where $k(n-1)$ does not depend on $j$.
\end{proposition}

\begin{proof}
	Let us rewrite \eqref{eq:ajder} as
	\begin{equation}\label{a1}
	a_j^{(r)}(u) = \frac{(-2)^r}{(S_j-S_{j-1})^r} \frac{\phi^{(r-1)} (\tau_j(u))}{C_\Phi}.
	\end{equation}
	In order to study the behavior of $\phi^{(r-1)}(\tau_j(u))$,
	let us start focusing on the function
	$s:\mathbb{R} \rightarrow \mathbb{R}$ given by
	\begin{equation*}
	s(t)= \begin{cases}
	e^{-\frac{1}{t}} & \mbox{ if } t >0 \\ 0 &\mbox{ otherwise.}
	\end{cases}
	\end{equation*}
	Since $s \in C^\infty \left(\mathbb{R}\right)$, we can explicitly compute its derivatives for any $n \in \mathbb{N}$ as
	\begin{equation*}
	s^{(n)}(t)= \begin{cases}
	\frac{G_n(t)}{t^{2n}} s(t) & \mbox{ if } t >0 \\ 0 &\mbox{ otherwise}
	\end{cases},
	\end{equation*}
	where $G_n$ is a polynomial of degree $n-1$ defined recursively by the following formula
	\begin{align*}
	G_1(t)&=1\\
	G_{n+1}(t)&=t^{2}G_{n}^{\prime}(t)-(2nt-1)G_{n}(t).
	\end{align*}
	Since
	\begin{equation*}
	\phi(\tau_j(u))= \begin{cases}
	e^{-\frac{1}{1-\tau_j(u)^2}} & \mbox{ if } \tau_j(u) \in [-1,1] \\ 0 &\mbox{ otherwise}
	\end{cases},
	\end{equation*}
we can rewrite
	$$\phi(\tau_j(u))= s(g(\tau_j(u))) \quad  \text{with } g(y)=1-y^2.$$
	Using the notation $D^{(n)}=\frac{d^n}{du^n}$, and applying the chain rule for high order derivatives for composite functions, the so-called \emph{Fa\`a di Bruno's formula} yields for $\tau_j(u) \in [-1,1]$,
\[
	D^{(n)}\phi(\tau_j(u))= n! \sum_{\nu=1}^{n} \frac{D^{(\nu)} s(g(\tau_j(u)))}{\nu!} \sum_{h_1+\dots+h_\nu=n} \frac{D^{h_1}(1-\tau_j(u)^2)}{h_1!}\dots  \frac{D^{h_\nu}(1-\tau_j(u)^2)}{h_\nu!}
\]
\[
	= n! \sum_{\nu=1}^{n} \frac{G_n(g(\tau_{j}(u)))}{g(\tau_j(u))^{2\nu}}\frac{s(g(\tau_j(u)))}{\nu!} \sum_{h_1+\dots+h_\nu=n}
	\frac{D^{h_1}(1-\tau_j(u)^2)}{h_1!}\dots  \frac{D^{h_\nu}(1-\tau_j(u)^2)}{h_\nu!}
\]
\[
	=n!
	\sum_{\nu=1}^{n}  \frac{G_n(1-\tau_{j}(u)^2)}{(1-\tau_j(u)^2)^{2\nu}}\frac{e^{-\frac{1}{1-\tau_j(u)^2}}}{\nu!} \sum_{h_1+\dots+h_\nu=n}
	\frac{D^{h_1}(1-\tau_j(u)^2)}{h_1!}\dots  \frac{D^{h_\nu}(1-\tau_j(u)^2)}{h_\nu!}
\]
where $h_i\geq 1$.\\
	
	Before we proceed further, we need to recall a couple of immediate facts. First note that
if $G_n$ is a polynomial of degree $n$, then
since $|\tau_j(u)|\leq 1$
$$
\left \vert G_n(1-\tau_j(u)^2)\right \vert  \leq C(n).
$$

Also, it holds that

	\begin{equation*}
	\sum_{h_1+\dots+h_\nu=n}\frac{D^{h_1}(1-\tau_j(u)^2)}{h_1!}\dots
\frac{D^{h_\nu}(1-\tau_j(u)^2)}{h_\nu!}  \leq \binom{n+\nu-1}{n} (2)^\nu
(\tau_j(u))^\nu.
	\end{equation*}
	Indeed inside the sum we have the first and second derivatives of $(1-\tau_j(u)^2)$ and hence we are
summing terms of the form $2^\alpha (2\tau_j(u))^\beta$  with $\alpha+\beta=\nu$. The binomial coefficient counts all the possible
combinations such that $h_1+\dots+h_\nu=n$. \\
	
	Thus we have that
	\begin{equation*}
	D^{(n)}\phi(\tau_j(u))\leq n!
	 C(n)e^{-\frac{1}{
1-\tau_j(u)^2}}\sum_{\nu=1}^{n}  \frac{\tau_j(u)^\nu }{(1-\tau_j(u)^2)^{2\nu}
}\frac{2^\nu}{\nu!} \binom{n+\nu-1}{n}.
	\end{equation*}

	Now, considering that
\[
	 \frac{\tau_j(u)^\nu }{
(1-\tau_j(u)^2)^{2\nu}} \leq \frac{1 }{(1-\tau_j(u)^2)^{2n}},
\]
\[
	\sum_{\nu=1}^{n} \frac{(2)^\nu}{\nu!} \binom{n+\nu-1}{n}=\frac{2^n n \binom{2n}{n}}{n+1};
\]
it follows that
	\begin{equation*}
	D^{(n)}\phi(\tau_j(u))\leq n!
	\frac{2^n n \binom{2n}{n}}{n+1}C(n) e^{-\frac{1}{1-\tau_j(u)^2}
} \frac{1 }{(1-\tau_j(u)^2)^{2n}}.
	\end{equation*}
	Finally, observe that
	
$$
\left|e^{-\frac{1}{1-\tau_j(u)^2}} \frac{1 }{(1-\tau_j(u)^2)^{2n}}
\right|\leq \max \bigg\{ e^{-\frac{1}{1-\tau_j(u)^2}} \frac{1 }{
(1-\tau_j(u)^2)^{2n}} \bigg\}= \frac{k(n)}{e^{2n}}
$$
for $\tau_j(u) \in [-1,1],$
leading to

\[
\left \vert D^{(n)}\phi(\tau_j(u)) \right \vert \leq k(n)
\]
where $k(n)$ does not depend on $j$. Substituting in (\ref{a1}) the proof of the proposition is completed.
\end{proof}

The next result is similar.

\begin{lemma}
For any $b_j(\cdot)$ defined as in \ref{sabato2} and $n=1,2,....$, we have that
\begin{equation*}
|D^{(n)}b_{j}(u)|\leq K(n)\frac{1}{(S_{j}-S_{j-1})^{n}}\text{ ,}
\end{equation*}
where $K(n)$ does not depend on $j$.
\end{lemma}

\begin{proof}
	We study
	$b_j(u)=\sqrt{a_{j+1}(u)-a_j(u)}$ in the interval $u \in [S_{j-1},S_{j+1}]$. Recalling (\ref{eq:bjexplicit}),

we focus first on $[S_{j-1}, S_j]$.	Again, Fa\`a di Bruno's formula implies
	\begin{equation*}
	D^{(n)}b_j(u)= n! \sum_{\nu=1}^{n} \frac{D^{(\nu)} \sqrt{1-a_j(u)}}{\nu!} \sum_{h_1+\dots+h_\nu=n} \frac{D^{h_1}(a_j(u))}{h_1!}\dots  \frac{D^{h_\nu}(a_j(u))}{h_\nu!}
		\end{equation*}
From Proposition \ref{bound1} it follows that
\[
|D^{(n)}b_j(u)|\leq
\]
\[
n! \sum_{\nu=1}^{n} \left|\frac{D^{(\nu)} \sqrt{1-a_j(u)}}{\nu!} \right|\sum_{h_1+\dots+h_\nu=n}\frac{C(h_1)}{h_1!} \left(\frac{2}{S_{j}-S_{j-1}}\right)^{h_1} \left(\frac{1}{1-\tau_j(u)^2}\right)^{2h_1} \dots \times
\]
\[
\times \frac{ C(h_\nu)}{h_\nu!} \left(\frac{2}{S_{j}-S_{j-1}}\right)^{h_\nu} \left(\frac{2}{1-\tau_j(u)^2}\right)^{2h_\nu}  \left(e^{-\frac{2}{1-\tau_j(u)^2}}\right)^\nu
\]
\[
\leq   n! C(n) \left(\frac{2}{S_{j}-S_{j-1}}\right)^{n}\sum_{\nu=1}^{n} \binom{n+\nu-1}{n}\left|\frac{D^{(\nu)} \sqrt{1-a_j(u)}}{\nu!} \right|\left( e^{-\frac{1}{1-\tau_j(u)^2}}\right)^\nu \left(\frac{1}{1-\tau_j(u)^2}\right)^{2n}
\]
\[
= n! C(n) \left(\frac{2}{S_{j}-S_{j-1}}\right)^{n}\sum_{\nu=1}^{n} \binom{n+\nu-1}{n} \left|\frac{1}{\nu!} \frac{1}{(1-a_j(u))^{\nu-1/2}}\right|\left( e^{-\frac{1}{1-\tau_j(u)^2}}\right)^\nu \left(\frac{1}{1-\tau_j(u)^2}\right)^{2n}.
\]
Now we have that
\[	
\left| \frac{e^{-\frac{1}{1-\tau_j(u)^2}}}{1-e^{-\frac{1}{1-\tau_j(u)^2}}}  \right| \leq \frac{1}{e-1} \text{ , } \left| \frac{e^{-\frac{1}{1-\tau_j(u)^2}}}{(1-\tau_j(u)^2)^{2n}}  \right| \leq \frac{k(n)}{e^{2n}}.
\]
Proceeding similarly in $[S_j, S_{j+1}]$, the thesis follows.
\end{proof}

\bigskip

Claudio Durastanti

Department S.B.A.I., Sapienza University of Rome

claudio.durastanti@uniroma.it

\bigskip

Domenico Marinucci

Department of Mathematics, University of Rome Tor Vergata

marinucc@mat.uniroma2.it

\bigskip

Anna Paola Todino

Department of Mathematical Sciences, Politecnico di Torino

anna.todino@polito.it

\end{document}